\documentclass[a4paper]{article}
\usepackage{amsmath}
\usepackage{amssymb}
\usepackage{dsfont}
\usepackage{amsthm}

\newtheorem{theorem}{Theorem}[section]

\newtheorem{lemma}[theorem]{Lemma}
\newtheorem{corollary}[theorem]{Corollary}
\newtheorem{proposition}[theorem]{Proposition}
\newtheorem{example}[theorem]{Example}

\newtheorem{question}[theorem]{Question}
\newtheorem{observation}[theorem]{Observation}
\newtheorem{conjecture}[theorem]{Conjecture}

\def\noi{\noindent}

\usepackage{tikz}
\usetikzlibrary{matrix,backgrounds,fit,trees}
\usetikzlibrary{through,calc}
\usetikzlibrary{decorations.pathreplacing}

\begin{document}

\title{Sharp conditions for the existence of an even $[a,b]$-factor in a graph}

\author{Eun-Kyung Cho\thanks{Department of Mathematics, Pusan National University,
Busan, 46241, Republic of Korea; ekcho@pusan.ac.kr.}
\and
Jong Yoon Hyun\thanks{Korea Institute for Advanced Study, Seoul, 02455, Republic of Korea;  hyun33@kias.re.kr.
Research supported by NRF-2017R1D1A1B05030707.}
\and
Suil O \thanks{Applied Mathematics and Statistics, The State University of New York, Korea, Incheon, 21985;  suil.o@sunykorea.ac.kr (corresponding author). Research supported by NRF-2017R1D1A1B03031758.}
\and
Jeong Rye Park\thanks{Finance.Fishery.Manufacture Industrial Mathematics Center on Big Data,  Pusan National University, Busan, 46241, Republic of Korea; parkjr@pusan.ac.kr. Research supported by NRF-2018R1D1A1B07048197.}
}

\date{}

\maketitle
\begin{abstract} {Let $a$ and $b$ be positive integers. An even $[a,b]$-factor of a graph $G$ is a spanning subgraph $H$ such that for every vertex $v \in V(G)$, $d_H(v)$ is even and $a \le d_H(v) \le b$. Matsuda conjectured that if $G$ is an $n$-vertex 2-edge-connected graph such that $n \ge 2a+b+\frac{a^2-3a}b - 2$, $\delta(G) \ge a$, and $\sigma_2(G) \ge \frac{2an}{a+b}$, then $G$ has an even $[a,b]$-factor. In this paper, we provide counterexamples, which are highly  connected. Furthermore, we give sharp sufficient conditions for a graph to have an even $[a,b]$-factor. For even $an$, we conjecture a lower bound for $\lambda_1(G)$ in an $n$-vertex graph to have an  $[a,b]$-factor, where $\lambda_1(G)$ is the largest eigenvalue of $G$.
}
\end{abstract}

\noi{\bf Keywords.}
Even $[a,b]$-factor;
edge-connectivity;
vertex-connectivity;
spectral radius.\\

\noi{\bf AMS subject classifications.}
05C70, 05C40, 05C50.

\vskip 1.5pc
\section{Introduction} 
\label{sec:introduction}

Throughout all sections, a graph $G$ is finite, simple, and undirected. 
We denote by $V(G)$ the set of vertices of $G$ and by $E(G)$ the set of edges of $G$.
For $S \subseteq V(G)$, we denote by $G-S$ the subgraph of $G$ obtained from $G$ by deleting
the vertices in $S$ together with the edges incident to vertices in $S$. For $S, T \subseteq V(G)$, we denote by $[S,T]$ the set of edges joining $S$ and $T$. The {\it degree} of a vertex $v$ in $G$, written $d_G(v)$ (or $d(v)$ if $G$ is clear from the context), is the number of edges in $E(G)$ incident to the vertex $v$. The {\it minimum} and {\it maximum degree} of a graph $G$ are denoted by $\delta(G)$ and $\Delta(G)$, respectively. For a subgraph $H$ of $G$, the degree of a vertex $v$ in $H$, written $d_H(v)$, is the number of edges in $E(H)$ incident to the vertex $v$.
An {\it even $[a,b]$-factor} of a graph is a spanning subgraph $H$ such that $d_{H}(v)$ is even
and $a \leq d_{H}(v) \leq b$ for all $v \in V(G)$. If $a=b$, then we call it an {\it $a$-factor}.
A graph $G$ is {\it $k$-edge-connected} if for $S \subseteq E(G)$ with $|S| < k$, $G-S$ is connected. The {\it edge-connectivity} of $G$, denoted $\kappa'(G)$, is the maximum $k$ such that $G$ is $k$-edge-connected. A graph $G$ is {\it $k$-vertex-connected} if $|V(G)|\ge k+1$ and for $S \subseteq V(G)$ with $|S| < k$, $G-S$ is connected. The {\it vertex-connectivity} of $G$, denoted $\kappa(G)$, is the maximum $k$ such that $G$ is $k$-vertex-connected.



Kouider and Vestaargard~\cite{KouVes},~\cite{KouVes2} had explored sufficient conditions for a graph to have an even $[a,b]$-factor. In 2005, Matsuda~\cite{Mat} gave a sharp sufficient condition for a graph to have an even $[2,b]$-factor and proposed a conjecture for the existence of an even $[a,b]$-factor in a graph as follows:

\medskip

\noindent
\begin{conjecture}\label{conj}
Let $2 \leq a \leq b$ be even integers.
If $G$ is a graph with $n$ vertices such that 
(i) $\kappa'(G) \ge 2$, 
(ii) $n \geq 2a+b+\frac{a^2-3a}{b}-2$, 
(iii) $\delta(G) \geq a$, and \\
(iv) $\sigma_2(G) \geq \frac{2an}{a+b}$, where $\sigma_2(G)=\min_{uv\not\in E(G)}\left(d(u)+d(v)\right)$, \\
then $G$ contains an even $[a,b]$-factor.
\end{conjecture}

However, Conjecture~\ref{conj} is not true even when $a=2$. Remark 3 in~\cite{Mat} says that if $n=b+2$,
then Conjecture~\ref{conj} does not hold. Theorem 8 in~\cite{Mat} says that if we replace $n \ge b+2$ by $n \ge b+3$, then $G$ contains an even $[2,b]$-factor. A result of Iida and Nishimura~\cite{IN} implies that Conjecture~\ref{conj} is true when $a=b$.

For $a \ge 4$, all other conditions in the conjecture are sharp, except $\kappa'(G) \ge 2$. In Section 2,  we provide counterexamples, which are $(a-1)$-edge-connected.
Furthermore, there are also $(a-1)$-vertex-connected graphs satisfying all conditions in Conjecture~\ref{conj}, which do not contain an even $[a,b]$-factor. Thus to guarantee the existence of an even $[a,b]$-factor in a graph, we need high vertex-(or edge-)connectivity. By reinforcing the condition $\sigma_2(G) \ge \frac{2an}{a+b}$ to $\delta(G) \ge \frac{an}{a+b}$, we give sufficient conditions for a graph to have an even $[a,b]$-factor in Theorem~\ref{main}.

\begin{theorem}[Main Theorem]\label{main}
Let $4 \leq a \leq b$ be even integers.
If $G$ is a graph with $n$ vertices
such that (i) $\kappa(G) \ge a$, (ii) $n \geq 2a+b+\frac{a^2-3a}{b}-2$, and (iii) $\delta(G) \geq \frac{an}{a+b}$,
then $G$ contains an even $[a, b]$-factor.
\end{theorem}

Katerinis~\cite{K}, and Egawa and Enomoto~\cite{EE} independently showed  that Theorem~\ref{main} is true when $a=b$.
In this paper, we prove for all $4 \le a \le b$ including the case $a=b$.
In the papers~\cite{K,EE}, to have an $[a,a]$-factor (or $a$-factor), one of the sufficient conditions is just ``connected''. However, if there is an enough gap between $a$ and $b$, then to have an even $[a,b]$-factor, a graph must be highly connected (See Section~\ref{sec:matsuda_conjecture}).

Note that Condition (ii) and (iii) in Theorem~\ref{main} imply that $\delta(G) \ge a+1$.
If $\delta(G) \le a$, then Condition (iii) says $\frac{an}{a+b} \le a$. Thus we have $n \le a+b$,
which contradicts Condition (ii). 

The examples in Section 2 show that the conditions in Theorem~\ref{main} are sharp. In Section 3, we prove Theorem~\ref{main} by using Corollary~\ref{cor} of Lovasz's $(g,f)$-factor Theory.

\begin{theorem}[Lovasz's $(g,f)$-factor Theory \cite{Lov}]\label{lov}
Let $G$ be a graph and let $g, f$ be two integer valued functions defined on $V(G)$
such that $0 \leq g(x) \leq f(x) \leq d_{G}(x)$ for all $x \in V(G)$.
Then $G$ has a $(g, f)$-factor if and only if
$$\sum_{v \in T} \left(d(v)-g(v)\right) + \sum_{u \in S} f(u) -|[S,T]|- q(S,T) \geq 0$$
for all disjoint subsets $S$ and $T$ of $V(G)$, where $q(S, T)$ is the number of components $Q$
of $G - (S \cup T)$ such that $g(v) = f(v)$ for all $v \in V(Q)$ and
\[
|[Q, T]| + \sum_{ v \in V(Q)} f(v) \equiv 1~ (\text{mod } 2).
\]
\end{theorem}

\begin{corollary}\label{cor}
Let $a$ and $b$ be even integers with $2 \leq a \leq b$.
A graph $G$ has an even $[a,b]$-factor if 
$$q(S,T) -b|S|+a|T|- \sum_{v \in T} d_{G - S}(v) \leq 0$$
for all disjoint choices $S, T \subseteq V(G)$,
where $q(S,T)$ is the number of components $Q$ of $G - (S \cup T)$
such that $|[Q,T]|$ is odd.
\end{corollary}

By applying Theorem~\ref{lov} when $g(x)=a$ and $f(x)=b$, we have Corollary~\ref{cor}.

We point out that Tutte~\cite{Tut2} proved that the Lovasz's $(g,f)$-factor Theory~\cite{Lov}
can be demonstrated by using Tutte's $f$-factor Theory~\cite{Tut1}.

The Parity Lemma is also used in the proof of our main result.

\begin{lemma}[Parity Lemma]\label{paritylemma}
Let $a$ and $b$ be positive integers with the same parity.
Then $q(S,T) -b|S|+a|T|- \sum_{v \in T} d_{G - S}(v)$ has the same parity as $a$ and $b$
for any disjoint sets $S, T \subseteq V(G)$.
\end{lemma}






\section{Sharp Examples} 
\label{sec:matsuda_conjecture}




In this section, by providing Example~\ref{examp1} and Example~\ref{examp2}, we show why high edge-(or vertex-)connectivity in Theorem~\ref{main} requires. Note that Matsuda~\cite{Mat} showed in the last section that Condition (ii) and (iii) in Theorem~\ref{main} are sharp.

Example~\ref{examp1} shows that if a graph satisfying Condition (ii), (iii), and (iv) in Conjecture~\ref{conj} is not $a$-edge-connected, then we cannot guarantee the existence of even $[a,b]$-factor in the graph. Thus the graph in Example~\ref{examp1} is a counterexample to Conjecture~\ref{conj}, which has edge-connectivity equal to $a-1$.

\begin{example}\label{examp1}{\rm
	Let $a$ and $b$ be even integers such that $12\leq 3a\leq b$,
	and let $t$ be an integer such that $t \geq\frac{(a+b)^2-3a-4b}{2b} (=\frac{2a+b-4}2+\frac{a^2-3a}{2b} > a)$.  For $i \in \{1,2\}$, let $H_i$ be a copy of the complete graph on $t$ vertices, and let $V(H_i)=\{x_{i1},\ldots,x_{it}\}$.  Let $H_3$ be a copy of the complete graph on 2 vertices and let $V(H_3)=\{y, z\}$. Suppose that $H$ is the graph obtained from $H_1$, $H_2,$ and $H_3$ by adding edges between $y$ and $x_{11},\ldots, x_{1(\frac a2 -1)}, x_{2 \frac a2},\ldots, x_{2(a-1)}$, and between $z$ and $x_{21},\ldots, x_{2(\frac a2 -1)},x_{1\frac a2},\ldots,$ $x_{1(a-1)}$
	(see Figure~\ref{Matsuda_edge_conn}).}
	
	\end{example}

	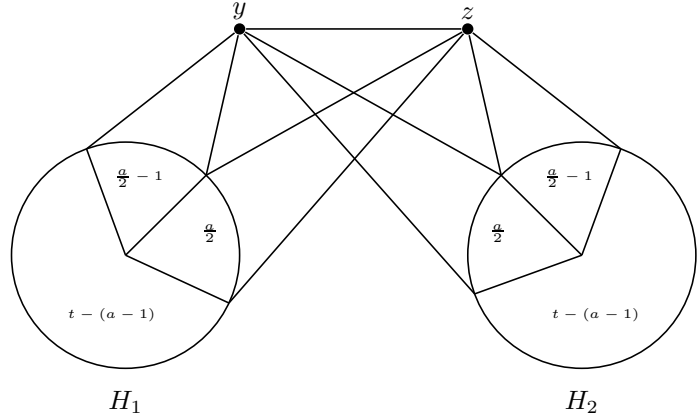
\begin{figure}[h]
	\centering
	\begin{tikzpicture}[auto,node distance=1cm,semithick,scale=1.5]
		\tikzstyle{vertex}=[circle,fill,inner sep=1.5pt]
		\node (a) 							  {};

		\node[vertex] (La) at (-1,0)		  {};
		\coordinate (LC) at (-2,-2) 		  {};
		\draw (LC) circle (1cm);
		
		\draw (LC) -- +(110:1) coordinate(L1);
		\draw (LC) -- +(45:1)  coordinate(L2);
		\draw (LC) -- +(-25:1) coordinate(L3);

		\node[vertex] (Ra) at (1,0) 	  	  {};
		\coordinate (RC) at (2,-2)  		  {};
		\draw (RC) circle (1cm);

		\draw (RC) -- +(70:1)  coordinate(R1);
		\draw (RC) -- +(135:1) coordinate(R2);
		\draw (RC) -- +(200:1) coordinate(R3);

		\draw (L1) -- (La);
        \draw (L2) -- (La);
        \draw (L2) -- (Ra);
        \draw (L3) -- (Ra);

        \draw (R1) -- (Ra);
        \draw (R2) -- (Ra);
        \draw (R2) -- (La);
        \draw (R3) -- (La);

        \path (La) edge (Ra);
        \node at (-2,-3.3) {$H_1$};
		\node at (2,-3.3)  {$H_2$};
		\draw (La) node[above] {$y$};
		\draw (Ra) node[above] {$z$};

		\draw (230:3.3) node {\tiny $t-(a-1)$};
		\draw (215:2.3) node {\tiny $\frac{a}{2}-1$};
		\draw (235:2.2) node {\tiny $\frac{a}{2}$};

		\draw (310:3.3) node {\tiny $t-(a-1)$};
		\draw (325:2.3) node {\tiny $\frac{a}{2}-1$};
		\draw (305:2.2) node {\tiny $\frac{a}{2}$};
	
	\end{tikzpicture}
	\caption{The graph $H$ in Example~\ref{examp1}}
	\label{Matsuda_edge_conn}
	\end{figure}
	
	\begin{proposition}
	The graph in Example~\ref{examp1} has edge-connectivity equal to $a-1$ and satisfies all conditions in Conjecture~\ref{conj}. Furthermore, it does not contain an even $[a,b]$-factor.
	\end{proposition}

	\begin{proof}  
	Since there are $\frac a2 -1$ edges between $y$ and $H_1$ and $\frac a2$ edges between $y$ and $H_2$, and $H_1$ and $H_2$ are both complete graphs, there are exactly $a-1$ edge-disjoint $y-z$ paths including the $yz$ edge. Also, since there are exactly $a-1$ edges between $H_i$ to $H_3$, we have $\kappa'(H)=a-1$.
	
	The order of $H$ is
	\[|V(H)|=2t+2\geq\frac{(a+b)^2-3a-4b}{b}+2=2a+b+\frac{a^2-3a}{b}-2.\]

	Since every vertex in $V(H_1) \cup V(H_2)$ has degree at least $a$ and $d_H(y)=d_H(z)=a$, we have  $\delta(H)=a$.

	Since $\sigma_2(H)=a+(t-1)$ and $t=\frac{|V(H)|-2}{2}$, we have
	\[\sigma_2(H)=a+\frac{|V(H)|}{2}-2
		\geq\frac{|V(H)|}{2}=\frac{2a|V(H)|}{4a}\geq\frac{2a|V(H)|}{a+b}.\]
	Thus $H$ satisfies all conditions in Conjecture~\ref{conj}.
	
	Now, we prove that $H$ does not contain an even $[a,b]$-factor.
	Assume to the contrary that $H$ has an even $[a,b]$-factor $F$. Since $d_H(y)=a$, all edges incident to $y$ must be in $F$. Since $H_1 \cap F$ is also a graph,
	$\sum_{v\in V(H_1\cap F)}d_{H_1\cap F}(v)$ must be even by the degree-sum formula.   Note that the $a-1$ edges incident to both $H_1$ and $H_3$ in $F$ are not in $H_1\cap F$. Thus we have 
	\[\sum_{v\in V(H_1\cap F)}d_{H_1\cap F}(v)
		=\sum_{v\in V(H_1\cap F)}d_{F}(v)-(a-1).\]
However, the degree sum is odd since $a-1$ is odd and every vertex in $F$ has even degree. Thus we have the desired result.

		\end{proof}

Example~\ref{examp2} shows that if a graph satisfying Condition (ii), (iii), and (iv) in Conjecture~\ref{conj} is not $a$-vertex-connected, then we cannot guarantee the existence of even $[a,b]$-factor in the graph. Thus the graph in Example~\ref{examp2} is also a counterexample to Conjecture~\ref{conj}, which is $(a-1)$-vertex-connected. Note that in Example~\ref{examp1}, we require $b \ge 3a$ while in Example~\ref{examp2}, we require $b \ge \frac{a^2-3a+a\sqrt{(a-3)(a+1)}}{2}$.
\begin{example}\label{examp2}{\rm
	Let $a$ and $b$ be even integers at least 4 with $b \geq \frac{a^2-3a+a\sqrt{(a-3)(a+1)}}{2}$. Let $L_0$ be the trivial graph on $a-1$ vertices, and let $V(L_0)=\{y_1,\ldots,y_{(a-1)}\}$. For $1 \le i \le a$, let $L_i$ be a copy of the complete graph on $a+2$ vertices and let $V(L_i)=\{x_{i1},\ldots,x_{i(a+2)}\}$. Let $t$ be a positive integer such that $(a+2 \le) -a^2-a+b+\frac{a^2-3a}b -1 \le t \le -a^2-2a+b+\frac ba +2$. Let $L_{a+1}$ be a copy of the complete graph on $t$ vertices and let $V(L_{a+1})=\{x_{(a+1)1},\ldots, x_{(a+1)t}\}$. Suppose that $L$ is the graph obtained from $L_0,\ldots, L_{a+1}$ by adding edges between $y_{j}$ and $x_{ij}$ for all $i \in \{1,\ldots,a+1\}$ and for all $j \in \{1,\ldots,a-1\}$
	(see Figure~\ref{Matsuda_edge_conn2}).}
	
	\end{example}

	\begin{proposition}\label{prop2}
	The graph in Example~\ref{examp2} has vertex-connectivity equal to $a-1$ and satisfies all conditions in Conjecture~\ref{conj}. Furthermore, it does not contain an even $[a,b]$-factor.
	\end{proposition}

	\begin{proof}  For each $i \in [a+1]$, there are $a-1$ vertex-disjoint paths between any vertex in $L_i$ and $L_0$ by using the vertex $x_{i1}, x_{i2}, \ldots, x_{i(a-1)}$. Also, for $i \neq j$, there are  $a+1$ vertex-disjoint paths between $y_i$ and $y_j$ by using the path $y_ix_{ki}x_{kj}y_j$ for $k \in [a+1]$. Thus we have $\kappa(L)=(a-1)$.

	Since $t \ge -a^2-a+b+\frac{a^2-3a}b -1$,  the order of $L$ is
	\[|V(L)|=a-1+(a+2)a+t \geq 2a+b+\frac{a^2-3a}{b}-2.\]

	Since for $i \in \{0,1,\ldots,a+1\}$, every vertex in $V(L_i)$ has degree at least $a+1$ and every vertex in $L_0$ has degree $a+1$, we have  $\delta(L)=a+1$.

	Since $\sigma_2(L)=2(a+1)$ and $t \le -a^2-2a+b+\frac ba +2$, we have
	\[\frac{2a|V(L)|}{a+b} = \frac{2a(a^2+3a-1+t)}{a+b} \le \frac{2a(a+b+\frac ba +1)}{a+b} = 2(a+1)
	=\sigma_2(L)\]
	Thus $F$ satisfies all conditions in Conjecture~\ref{conj}.
	
	Now, we prove that $L$ does not contain an even $[a,b]$-factor. Assume to the contrary that $L$ has an even $[a,b]$-factor $F$. Then we have $d_F(v)=a$ for every vertex in $L_0$ since $a$ is even. Since $\sum_{v \in V(L_i \cap F)} d_{L_i\cap F}(v)$ must be even by the degree-sum formula, there are at most $a-2$ edges coming out from $V(L_i)$ in $F$. Thus we have $$a(a-1) \le (a-2)(a+1),$$ which is a contradiction.
		\end{proof}
		
		Proposition~\ref{prop2} shows that Condition (i) in Theorem~\ref{main} is tight.

\begin{figure}[h]
\begin{center}
	\begin{tikzpicture}[auto,node distance=1cm,semithick,scale=1.5]
		\tikzstyle{vertex}=[circle,fill,inner sep=1.5pt]
	
		\node   	  (1)							{$\cdots$};
		\node		  (0) [xshift=15mm,yshift=35mm] {$\cdots$};
		\node 		  (b) [right of=0]  			{$\cdots$};
		\node[vertex] (c) [right of=b]  			{};
		\node[vertex] (d) [left of=0]   			{};
		\node[vertex] (e) [left of=d]   			{};
		
		\draw (e)  node[above, scale=1]      	 {$y_1$};
		\draw (d)  node[above, scale=1]        	 {$y_2$};
		\draw (c)  node[above, scale=1]    {$y_{(a-1)}$};
		
		\node (C1) [xshift=-3cm] 			     {$L_1$};
		\node (C2) [xshift=+3cm] 		         {$L_a$};
		\node (C3) [xshift=+6cm]             {$L_{a+1}$};

		\draw (C1) circle [radius=0.7cm];
		\draw (C2) circle [radius=0.7cm];
		\draw (C3) circle [radius=0.7cm];

		\draw (C1) +(90:0.7) coordinate (L1);
		\draw (C1) +(70:0.7) coordinate (L2);
		\draw (C1) +(50:0.7) coordinate (L3);
		
		\draw (C2) +(130:0.7) coordinate (R1);
		\draw (C2) +(100:0.7) coordinate (R2);
		\draw (C2) +(70:0.7)  coordinate (R3);

		\draw (C3) +(130:0.7) coordinate (K1);
		\draw (C3) +(110:0.7) coordinate (K2);
		\draw (C3) +(90:0.7)  coordinate (K3);

		\path[-]
		(L1) edge (e)
		(L2) edge (d)
		(L3) edge (c);

		\path[-]
		(R1) edge (e)
		(R2) edge (d)
		(R3) edge (c);

		\path[-]
		(K1) edge (e)
		(K2) edge (d)
		(K3) edge (c);

	\end{tikzpicture}
	\caption{The graph $L$ in Example~\ref{examp2}}
	\label{Matsuda_edge_conn2}
\end{center}
\end{figure}
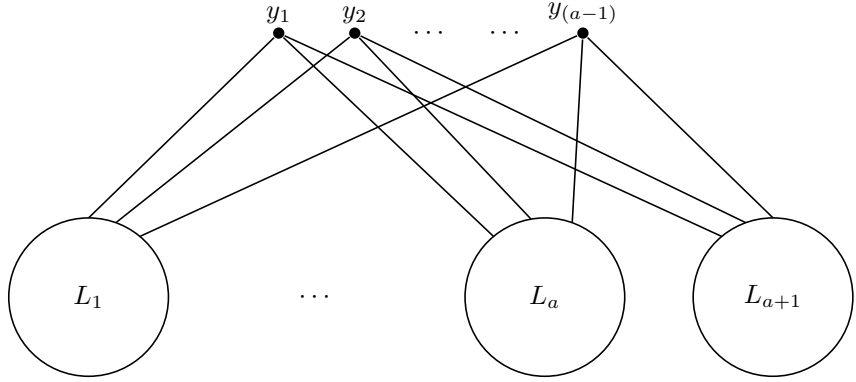


\section{Proof of Theorem~\ref{main}}

When $a=2$, Conjecture~\ref{conj} is true if we replace Condition (ii) by $n \ge b+3$.
So is Theorem~\ref{main} for $a=2$.
There is a counterexample to Conjecture~\ref{conj} if $n=b+2$ (see Remark 3~\cite{Mat}). 
From now, we assume that $a \ge 4$. The examples in Section 2 and the last section in~\cite{Mat} say that we require the conditions in Theorem~\ref{main} for a graph to have an even $[a,b]$-factor.

In this section, we prove Theorem~\ref{main}. Note that for $a=b$, Theroem~\ref{main} is true by Katerinis~\cite{K}, Egawa and Enomoto~\cite{EE}, and Iida and Nishimura~\cite{IN}. In this paper, we prove for all $4 \le a \le b$ including the case $a=b$.
To prove Case 3 and Case 4-1 in the proof of Theorem~\ref{main}, we use Proposition~\ref{f(x)}.

\begin{proposition}\label{f(x)}
Let $a, b, n,$ and $p$ be integers such that $4 \leq a \leq b$ and $p>0$, and 
let $f(x)=n+(a-1-\frac{an}{a+b})x+(x-1-b)\frac{ax-p}b$.

\noindent
(i) If $n \geq 2a+b+\frac{a^2-3a}{b}-2$,
then $f(b+1) <0$ and $f(a+b-3)<0$.

\noindent
(ii) If $n \geq 2a+b+\frac{a^2-3a}{b}+1$,
then $f(a+b-1)<0$ and $f(a+b-2)<0$.
\end{proposition}

\begin{proof}
(i) Assume that $n \geq 2a+b+\frac{a^2-3a}{b}-2$. Then we have
\begin{align*}
f(b+1)&=n+(a-1-\frac{an}{a+b})(b+1)=(1-a)(\frac{bn}{a+b}-b-1)\\
&\leq (1-a)\left[\frac{b}{a+b}(2a+b+\frac{a^2-3a}{b}-2)-b-1\right]\\
&=(1-a)\left[\frac{b(a-3)+a(a-4)}{a+b}\right] < 0
\end{align*}
and
\begin{align}
\nonumber f(a+b-3) & = n+(a-1- \frac{an}{a+b})(a+b-3) + [(a+b-3)-1-b] \frac{a(a+b-3)-p}{b} \\
\nonumber & = (4a+b-a^2-ab)(\frac{n}{a+b}-\frac{a+b-3}{b})-\frac{(a-4)p}{b}\\
\nonumber & \leq (4a+b-a^2-ab)(\frac{2a + b + \frac{a^2-3a}{b}-2}{a+b}-\frac{a+b-3}{b})-\frac{(a-4)p}{b}\\
\nonumber & = \frac{a(4-a)+b(1-a)}{a+b}-\frac{(a-4)p}{b}<0.
\end{align}

\noindent
(ii) Assume that $n \geq 2a+b+\frac{a^2-3a}{b}+1$. Then we have
\begin{align}
\nonumber f(a+b-1) & = n+(a-1- \frac{an}{a+b})(a+b-1) + [(a+b-1)-1-b] \frac{a(a+b-1)-p}{b} \\
\nonumber & = (2a+b-a^2-ab)(\frac{n}{a+b}-\frac{a+b-1}{b})-\frac{(a-2)p}{b}\\
\nonumber & \leq (2a+b-a^2-ab)(\frac{2a + b + \frac{a^2-3a}{b}+1}{a+b}-\frac{a+b-1}{b})-\frac{(a-2)p}{b}\\
\nonumber & = \frac{(a(2-a)+b(1-a))(-2a+2b)}{b(a+b)}-\frac{(a-2)p}{b}<0.
\end{align}
Since $f(x)$ is a quadratic function which has a positive leading coefficient
and $f(b+1) <0$ by (i),
we have $f(x)<0$ for all $x \in [b+1, a+b-1]$ so that $f(a+b-2)<0$.
\end{proof}

Now, we are ready to prove Theorem~\ref{main}.

\begin{proof}[Proof of Theorem~\ref{main}]
Assume to the contrary that $G$ has no even [$a$, $b$]-factor.
Then there exist disjoint subsets $S$ and  $T$ in $V(G)$ such that
\[0<q(S,T)-b|S|+a|T|-\sum_{v \in T}d_{G - S} (v)\]
by Corollary~\ref{cor}.
Let $p=-b|S|+a|T|$ so that
\[0<q(S,T)+p-\sum_{v \in T}d_{G - S} (v).\]
Note that $p>0$ since $q(S,T) -\sum_{v \in T}d_{G - S} (v) \leq 0$.

We consider four cases depending on $|T|$.

\begin{center}
	\begin{tikzpicture}
\draw (0,0)--(8,0);
\node at (0,0) {$\mid$};
\node at (-0.6,-0.35) {$|T| =$};
\node at (0.2,-0.35) {$\frac {b}a|S|$};
\node [above] at (1,0) {Case 2};
\node at (2,0) {$\mid$};
\node at (2,-0.35) {$b$};
\node [above] at (3,0) {Case 3};
\node at (4,-0.35) {$a+b-3$};
\node at (4,0) {$\mid$};
\node at (6,-0.35) {$a+b-1$};
\node at (6,0) {$\mid$};
\node [above] at (5,0) {Case 4};
\node at (8,-0.35) {$n$};
\node [above] at (7,0) {Case 1};
\node at (8,0) {$\mid$};
\end{tikzpicture}
\end{center}

In Case 4, we consider two subcases Case 4-1 and Case 4-2 depending on $n$.
To prove Case 1, Case 3, and Case 4-1, we use the same argument as in \cite{KouVes}.
For Case 2 and Case 4-2, we prove using a new technique.

{\it Case 1: $|T| \geq a+b$.}
Since $n \geq |S|+|T|+q(S,T)$, 
we have
\begin{align}\label{eq1}
\nonumber
&|S| =\frac{a|T|-p}{b} \leq \frac{a(n-|S|-q(S,T))-p}{b}
\iff |S|\leq \frac{a(n-q(S,T))-p}{a+b}\\
&\implies  |[S,T]| \leq |S||T| \leq \frac{a(n-q(S,T))-p}{a+b}|T|.
\end{align}
With Inequality~(\ref{eq1}), we have
\begin{align*}
 0&<q(S,T)-b|S|+a|T|-\sum_{v \in T}d_{G - S} (v)
 =q(S,T)+p-\sum_{v \in T}d_{G} (v)+|[S,T]|\\
 &\leq q(S,T)+p-\delta(G)|T| + \frac{a(n-q(S,T))-p}{a+b}|T|\\
 &\leq q(S,T)+p-\frac{an}{a+b}|T| + \frac{a(n-q(S,T))-p}{a+b}|T|\\
 &=q(S,T)+p - \frac{aq(S,T)+p}{a+b}|T|  \leq (1-a)q(S,T)< 0,
\end{align*}
which is a contradiction.


{\it Case 2: $|T| \leq b.$} Since $\delta(G)\ge a+1$, we have
\begin{align*}
 \nonumber 0&<q(S,T)-b|S|+a|T|-\sum_{v \in T} d_{G - S}(v)
 = q(S,T)-b|S|+a|T|-\sum_{v \in T} d_{G}(v)+|[S,T]|\\
 \nonumber &\leq q(S,T)-b|S|+a|T|-\delta(G)|T|+|S||T|
 \leq q(S,T)-b|S|+\left[ a-(a+1)+|S| \right] |T| \\
\nonumber &\leq q(S,T) - b|S| + b|S|-|T| 
 =q(S,T)-|T|,
\end{align*}
which implies $q(S,T) > |T| \geq 0.$ 
Let $l$ be the minimum of $|[Q,T]|$ over all components $Q$ of $G-(S\cup T)$ such that $|[Q,T]|$ is odd.
Then we have $l \geq 1$.
Also, we have $\sum_{v \in T} d_{G - S}(v) \geq lq(S,T)
\iff \frac{1}{l} \sum_{v \in T} d_{G - S} (v) \geq q(S,T)$.
Thus we have

\begin{align}\label{eq2}
\nonumber 0&<q(S,T)-b|S|+a|T|-\sum_{v \in T} d_{G - S}(v)
\leq 
\frac{1-l}{l}\sum_{v \in T} d_{G - S}(v)-b|S|+a|T|\\
\nonumber&= \frac{1-l}{l}\sum_{v \in T} d_{G}(v)-\frac{1-l}{l}|[S,T]|-b|S|+a|T|\\
\nonumber &\leq \frac{1-l}{l}\delta(G)|T|-\frac{1-l}{l}|S||T|-b|S|+a|T| \\
\nonumber&\leq \frac{(1-l)(a+1)}{l}|T|+|S|(\frac{l-1}{l}|T|-b)+a|T|\\
&\leq\frac{a+1-l}{l}|T|-|S|\frac{b}{l}.
\end{align}

If $a+1-l \leq 0$ in Inequality~(\ref{eq2}), then it is a contradiction.
Thus $a+1-l>0$, and since $|T| \le b$, we have 

\begin{align}
\nonumber 0 &< \frac{a+1-l}{l}|T|-|S|\frac{b}{l} 
\leq 
\frac{a+1-l-|S|}{l}b.
\end{align}

This gives $a+1-l-|S|>0$ so that $|S| \leq a-l$.

\medskip
\noindent
{\it Claim 1: $|S| = a-l$.}
Assume to the contrary that $|S| <a-l$.
Then each component $Q$ in $G-S-T$ with $|[Q,T]|=l$ can have at most $a-1$ neighbors in $S \cup T$.
Since $G$ is $a$-vertex-connected, there is only one such a component and every vertex of $T$ must have a neighbor in $Q$. Thus $|T| \le l$.
For $v \in V(T)$,
we have 
$$d(v) = |[\{v\}, S]|+\left(d_{G[T]}(v) +|[\{v\}, Q]|\right) \leq (a-l-1) + l = a-1 < a+1 \leq \delta(G),$$
which is a contradiction.
Thus $|S| = a-l$, which implies $S\neq \emptyset$ since $a$ is even and $l$ is odd.
Since $|S| \ge 1$ and $|T|\le b$, we have
\begin{align} \label{eq3}
\nonumber 0&<q(S,T)-b|S|+a|T|-\sum_{v \in T} d_{G - S}(v) \\
\nonumber &= q(S,T)-b|S|+a|T|-\sum_{v \in T} d_{G}(v)+|[S,T]|\\
\nonumber &\leq q(S,T)-b|S|+a|T|-\delta(G)|T|+|S||T|\\
\nonumber &\leq q(S,T)-b|S|+[a-(a+1)+|S|]|T|\\
&\leq q(S,T)-b|S|+(|S|-1)b 
=q(S,T)-b.
\end{align}
By Inequality~(\ref{eq3}), we have $q(S,T) \geq b+1$. Let $q(S,T)=b+\alpha$ for some $\alpha \geq1$.

Let $q_l$ be the number of components $Q$ of $G - (S \cup T)$ such that
$|[Q,T]|=l$. Since $|S|=a-l$, $|T| \le b$, and $q(S,T)=b+\alpha$, we have

\begin{align}\label{eq4}
 \nonumber 0&<q(S,T)-b|S|+a|T|-\sum_{v \in T} d_{G - S}(v)\\
\nonumber  &\leq q(S,T)-b(a-l)+a|T|-[lq_l +(l+2)(q(S,T)-q_l)]\\
 \nonumber &\leq (-l-1)q(S,T)+bl+2q_l
 = (-l-1)(b+\alpha)+bl+2q_l \\
 &=-b-(l+1)\alpha+2q_l.
\end{align}

By Inequality~(\ref{eq4}), we have $q_l > \frac{b+(1+l)\alpha}{2}$.
Note that $b$ and $1+l$ are even integers so that $\frac{b+(1+l)\alpha}{2}$ is an integer.
Thus $q_l \geq \frac{b+(1+l)\alpha}{2}+1$.

Let $m$ be the minimum of $|Q|$ over all components $Q$ in $G-(S \cup T)$ such that $|[Q,T]|=l$. 
There exists a vertex $v$ in $V(Q)$ such that $|[\{v\}, T]| \leq \frac{l}{m}$ by the pigeonhole principle.
Then we have
\begin{align}
\nonumber & \delta(G) \leq d(v)=d_Q(v)+|[\{v\},T]|+ |[\{v\},S]| \leq
(m-1)+(\frac{l}{m})+(a-l)\\
\nonumber \implies & m \leq \frac{\delta(G)+l+1-a-\sqrt{(\delta(G)+l+1-a)^2-4l}}{2} \\
\nonumber & \textrm{ or }
m \geq \frac{\delta(G)+l+1-a+\sqrt{(\delta(G)+l+1-a)^2-4l}}{2}.
\end{align}
Note that we have
\begin{align}\label{eq5}
\frac{\delta(G)+l+1-a-\sqrt{(\delta(G)+l+1-a)^2-4l}}{2}<1
 \iff \delta(G) \geq a+1.\end{align}
 
Since $m \ge 1$, we have $m \geq \frac{\delta(G)+l+1-a+\sqrt{(\delta(G)+l+1-a)^2-4l}}{2}$ by Inequality~(\ref{eq5}).
Note that we have
\begin{align}\label{eq6}
\nonumber &\frac{\delta(G)+l+1-a+\sqrt{(\delta(G)+l+1-a)^2-4l}}{2} \geq \delta(G)-|S| \\
\nonumber &\iff (\delta(G)+l+1-a)^2-4l \geq (\delta(G)-2|S|-1+a-l)^2\\
\nonumber &\iff (\delta(G)-|S|)(1-a+l+|S|) \geq l \\
&\iff \delta(G)-a+l \geq l 
 \iff \delta(G) \geq a.
\end{align}
By Inequality~(\ref{eq6}), we have $m \geq \delta(G)-|S|$, implying $n \geq |S|+|T|+q_l(\delta(G) -|S|)$.
Since $q_l \geq \frac{b+(1+l)\alpha}{2}+1$, $\delta(G) \geq \frac{an}{a+b}$, $2b-a(b+(1+l)\alpha) = (2-a)b-(1+l)\alpha <0$, and $a|T|-b|S| >0$,
we have
\begin{align*}
\nonumber &\frac{a+b}{a}\delta(G) \geq n 
\geq |S|+|T|+[\frac{b+(1+l)\alpha}{2}+1](\delta(G)-|S|)\\
&\iff \delta(G) \leq \frac{2a|T|-a[b+(1+l)\alpha]|S|}{2b -a[b+(1+l)\alpha]}
=\frac{2(a|T|-b|S|)}{2b -a[b+(1+l)\alpha]} +|S| <|S|,
\end{align*}
which is a contradiction.

{\it Case 3: $b+1 \leq |T| \leq a+b-3$.} Since $q(S,T) \le n-|S|-|T|$ and $
\delta(G) \ge \frac{an}{a+b}$, we have
\begin{align}\label{eq7}
\nonumber 0&<q(S,T)-b|S|+a|T|-\sum_{v \in T}d_{G - S} (v)\\
\nonumber &\leq (n-|S|-|T|)-b|S|+a|T|-\frac{an}{a+b}|T| + |S||T|\\
\nonumber &=n+(a-1-\frac{an}{a+b})|T|+(|T|-1-b)|S|\\
&=n+(a-1-\frac{an}{a+b})|T|+(|T|-1-b)\frac{a|T|-p}{b}.
\end{align}
Let $f(|T|)=n+(a-1-\frac{an}{a+b})|T|+(|T|-1-b)\frac{a|T|-p}{b}$.
Since $f$ is a quadratic function which has a positive leading coefficient,
the maximum value of $f$ occurs when $|T|=b+1$ or $|T|=a+b-3$.
By Proposition \ref{f(x)}, both $f(b+1)$ and $f(a+b-3)$ are negative, which contradicts
Inequality~(\ref{eq7}).

\medskip

{\it Case 4: $|T|=a+b-2$ or $a+b-1$.}

\quad {\it Case 4-1: $n\geq2a+b+\frac{a^2-3a}{b}+1$.}
By using the same argument with Case 3 and Proposition \ref{f(x)},
we have the desired result.

\quad {\it Case 4-2: $ 2a+b+\frac{a^2-3a}{b}-2 \le n < 2a+b+\frac{a^2-3a}{b}+1$.}
Let $|T| = a+b-k$ and $2a+b+\frac{a^2-3a}{b} -j \leq n < 2a+b+\frac{a^2-3a}{b} -j+1$
where $k \in \{1,2\}$ and $j \in \{0, 1, 2\}$.
Let $n=2a+b+\frac{a^2-3a}{b}-j+\epsilon$, where $0 \leq \epsilon < 1$.

\medskip
\noindent
{\it Claim 2: If $\delta(G) \geq j -k+i +n - |T| = i +a+\frac{a^2-3a}{b} +\epsilon$, 
then $a(3-k)-\epsilon b \geq (a-k)(j-k) + i(a+b-k) + (a-k-1)q(S,T) +2$, where $i$ is an integer.}
By Lemma \ref{paritylemma}, we have
\begin{align*}
 2 &\leq q(S,T) - b|S| +a|T|-\sum_{v \in T}d_G(v) +|[S,T]| \\
&\leq q(S,T)-b|S|+a|T|-\delta(G)|T|+|S||T|\\
 &\leq q(S,T)-b|S|+a|T|-(i +a+\frac{a^2-3a}{b} +\epsilon)|T|+|S||T|\\
&\leq q(S,T) +(a-k)(n-|T|-q(S,T))-(i +\frac{a^2-3a}{b} +\epsilon)(a+b-k)\\
&=(k+1-a)q(S,T) +(a-k)(a-i-j+k)-b(i+\frac{a^2-3a}{b} +\epsilon)\\
&=(k+1-a)q(S,T) +(a-k)(k-j)-i(a+b-k)+a(3-k)-\epsilon b.
\end{align*}
Thus we have the desired result. 

Since $\delta(G) \geq \frac{an}{a+b}$, 
we have
\begin{align}
\nonumber \delta(G) -n +|T| &\geq \frac{an}{a+b} -n+|T| =-\frac{bn}{a+b}+(a+b-k)\\
\nonumber &= \frac{-b(2a+b+\frac{a^2-3a}{b}-j+\epsilon)+(a+b-k)(a+b)}{a+b} \\
 & = \frac{(3-k)a+(j-k)b - \epsilon b}{a+b} \label{eq8} \\
\nonumber & > j-k-1,
\end{align}
which is true for $j \in \{0,1,2\}$ and $\epsilon \in [0,1)$.
Since $\delta(G) -n +|T|$ is an integer, we obtain $\delta(G) -n +|T| \geq j-k$,
which satisfies the condition on $\delta(G)$ when $i=0$ in Claim 2.
Thus  we have $a(3-k)-\epsilon b \geq (a-k)(j-k) + (a-k-1)q(S,T) +2$.
By Inequality~(\ref{eq8}), we have
\begin{align*}
\delta(G) -n +|T| &\geq \frac{(3-k)a+(j-k)b - \epsilon b}{a+b} \\
&\geq \frac{(j-k)b+(a-k)(j-k)  + (a-k-1)q(S,T) +2}{a+b}\\
&=j-k+\frac{-k(j-k) + (a-k-1)q(S,T) +2}{a+b} > j-k,
\end{align*}
which is true for $k \in \{1,2\}$ and $j \in \{0, 1, 2\}$ and $a\geq 4$.
Since $\delta(G) -n +|T|$ is an integer, we obtain $\delta(G) -n +|T| \geq j-k+1$,
which satisfies the condition on $\delta(G)$ when $i=1$ in Claim 2.
Thus  we have 
\begin{align}\label{eq10}
a(3-k)-\epsilon b \geq (a-k)(j-k) + (a+b-k) + (a-k-1)q(S,T) +2.
\end{align}

When $k=j=1$, Inequality~(\ref{eq10}) becomes $a-(\epsilon +1)b \geq (a-2)q(S,T) +1$
which is a contradiction since $a-(\epsilon +1)b \leq 0$ and $(a-2)q(S,T) +1 > 0$.
Similarly, we have a contradiction when $(k,j)\in\{(1,2),(2,1),(2,2)\}$
by using Inequality~(\ref{eq10}).
The remaining case is when $j=0$. By Inequality~(\ref{eq8}) and~(\ref{eq10}), we improve $\delta(G)$ as follows:
\begin{align*}
\delta(G) -n +|T| &\geq \frac{(3-k)a+(j-k)b - \epsilon b}{a+b} \\
&\geq \frac{(j-k)b+(a-k)(j-k) + (a+b-k) + (a-k-1)q(S,T) +2}{a+b}\\
&=j-k+1+\frac{-k(j-k+1) + (a-k-1)q(S,T) +2}{a+b} > j-k+1,
\end{align*}
which is true for $(k,j) \in \{(1,0),(2,0)\}$ and $a\geq 4$.
Since $\delta(G) -n +|T|$ is an integer, we obtain $\delta(G) -n +|T| \geq j-k+2$,
which satisfies the condition on $\delta(G)$ when $i=2$ in Claim 2.
Thus  we have 
\begin{align}\label{eq11}
a(3-k)-\epsilon b \geq (a-k)(j-k) + 2(a+b-k) + (a-k-1)q(S,T) +2.
\end{align}

When $k=1$ and $j=0$, Inequality~(\ref{eq11}) becomes $a-(\epsilon +2)b \geq (a-2)q(S,T) +2$
which is a contradiction since $a-(\epsilon +2)b < 0$ and $(a-2)q(S,T) +2 > 0$.
Similarly, we get a contradiction when $k=2$ and $j=0$,
which completes the proof.



\end{proof}

\section{Concluding Remarks}

In this section, we provide some questions and conjecture.

\begin{question}
If we replace ``$\kappa(G)$'' in Theorem~\ref{main} by ``$\kappa'(G)$'', then do we have the same conclusion?
\end{question}

\begin{question}
If we replace ``$\delta(G) \ge \frac{an}{a+b}$'' in Theorem~\ref{main} by ``$\sigma_2(G) \ge \frac{2an}{a+b}$'', then do we have the same conclusion?
\end{question}

We might be also interested in some sufficient conditions for a certain eigenvalue in a certain graph $G$ to have an even $[a,b]$-factor. If $G$ has an even $[a,b]$-factor, then we have $\lambda_1(G) \ge a$ since $\lambda_1(G) \ge \delta(G)$, where $\lambda_1(G)$ is the largest eigenvalue of $G$. Is there a sufficient condition for $\lambda_1(G)$ in a graph $G$ to have an even $[a,b]$-factor? 
If we restrict our attention to a complete bipartite graph, which looks the simplest case, then it is easy to get a sufficient condition for the largest eigenvalue.

\begin{observation} \label{obs}
	Let $G$ be the complete bipartite graph $K_{x,n-x}$ such that $n \ge 2x > 0$. 
	Then 
 $G$ has an $[a,b]$-factor if and only if $$\lambda_1(G) \ge \begin{cases} \sqrt{a(n-a)} & \text{ if } n < a+b \\ \quad \frac{\sqrt{ab}}{a+b}n & \text{ if } n \ge a+b. \end{cases}$$
 
\end{observation}
\begin{proof}  $G$ has an $[a,b]$-factor $F$
	  if and only if $$x \ge a \text{ and } (n-x-b)x \le (n-x)(x-a) (\Leftrightarrow x \ge \frac{an}{a+b})$$ since $\delta(F) \ge a$ and $\Delta(F) \le b$.
	
Thus we have the desired result with $\lambda_1(G)=\sqrt{x(n-x)}$.
\end{proof}

Among $n$-vertex graphs $G$ without $[a,b]$-factor, we guess that the $n$-vertex graph $H_{n,a}$ obtained from one vertex and a copy of $K_{n-1}$ by adding $a-1$ edges between them has the largest eigenvalue. Note that there are $n-a$ vertices with degree $n-2$, $a-1$ vertices with degree $n-1$, and 1 vertex with degree $a-1$ in the graph $H_a$. Thus $H_a$ cannot have an $[a,b]$-factor.

\begin{conjecture}
	Let $an$ be an even integer at least $2$, where $n \ge a+1$, and let $\rho(n,a)$ be the largest eigenvalue of $H_{n,a}$. If $G$ is an $n$-vertex graph with $\lambda_1(G) > \rho(n,a),$ then $G$ has an $[a,b]$-factor.
\end{conjecture}

We mention that $\lambda_1(H_{n,a})$ equals the largest root of $x^3-(n-3)x^2-(a+n-3)x-a^2+(a-1)n+1=0$ without giving a reason in detail.


\end{document}